\documentclass[11pt]{article}
\usepackage{amsmath}
\usepackage{amsfonts}
\usepackage{amsthm}
\usepackage{mathrsfs}
\usepackage{amssymb, setspace}
\usepackage{color,graphicx, epsfig, geometry, hyperref,fancyhdr}
\usepackage{float}
\newtheorem{theorem}{Theorem}[section]

\newtheorem{lemma}{Lemma}[section]

\newcommand{\N}{\mathbb{N}}
\newcommand{\Z}{\mathbb{Z}}
\newcommand{\weakc}{\rightharpoonup}
\newcommand{\R}{\mathbb{R}}
\newcommand{\C}{\mathbb{C}}

\newcommand{\dnu}{\partial_\nu}

\newcommand{\grad}{\nabla}

\graphicspath{{paper-pics/}}

\begin{document}
\setlength{\parskip}{1mm}
\setlength{\oddsidemargin}{0.1in}
\setlength{\evensidemargin}{0.1in}
\lhead{}
\rhead{}
\rfoot{}
\lfoot{}

{\bf  \Large \noindent Detecting inclusions with a generalized impedance condition from electrostatic data via sampling}

\begin{center}
Isaac Harris  \\
Department of Mathematics\\
Purdue University \\
West Lafayette, IN 47907\\
E-mail: harri814@purdue.edu
\end{center}

\begin{abstract}
\noindent In this paper, we derive a Sampling Method to solve the inverse shape problem of recovering an inclusion with a generalized impedance condition from electrostatic Cauchy data. The generalized impedance condition is a second order differential operator applied to the boundary of the inclusion. We assume that the Dirichlet-to-Neumann mapping is given from measuring the current on the outer boundary from an imposed voltage. A simple numerical example is given to show the effectiveness of the proposed inversion method for recovering the inclusion. We also consider the inverse impedance problem of determining the material parameters from the Dirichlet-to-Neumann mapping assuming the inclusion has been reconstructed where uniqueness for the reconstruction of the coefficients is proven.
\end{abstract}


\section{Introduction}\label{intro}
In this paper, we consider an inverse boundary value problem in electrostatic imaging. We propose using a {\color{black}Sampling (also known as qualitative) Method} to detect an inclusion with a generalized impedance boundary condition. Using the voltage and current measurements on the exterior boundary we will derive an algorithm for recovering the inclusions with little to no a prior information about the inclusion, which is one of the strengths of Sampling Methods. This means that one does not need to know the number of inclusions or have any estimate for the coefficients. These methods allow the user to reconstruct regions by deriving an `indicator' function from the measured data. This idea was first introduced in \cite{CK}. In particular, we will use that knowledge of the Dirichlet-to-Neumann mapping for Laplace's equation in a domain with inclusions. We assume that these subregions are impenetrable, where the electrostatic potential satisfies a generalized impedance boundary condition on the boundary of the inclusions. The electrical impedance tomography problem of visualizing/recovering the defective subregions from boundary measurements has many applications. 

The generalized impedance boundary condition can be used to model complex features such as coating and corrosion. The analysis for recovering a obstacle with a more general class of boundary conditions has been studied in \cite{FM-GIBC} where the Factorization Method was used to solve the inverse shape problem. In \cite{delamination} a generalized impedance condition is derived to asymptotically describe  delamination. Therefore, our method can be used to detect complex regions in electrostatic imaging. We will consider the inverse shape and inverse impedance problems. Our method for solving the inverse shape problem will be to recover the boundary via a Sampling Method that is of similar flavor to the work done in \cite{FM-EIT}. See monograph \cite{CCbook,CCMbook,kirschbook} and the references therein for the application of Sampling Methods to acoustic and electromagnetic scattering. Sampling Methods recover unknown obstacles by considering an ill-posed problem that involves the data operator and a singular solution to the background equation (i.e. without an inclusion/obstacle). 
The authors of \cite{LSM-imp-inclusion} used the {\color{black} Linear Sampling Method} to recover an impenetrable subregion of an inhomogeneous media using far field data. Recently these methods have been extended to problems in the time domain. In \cite{MUSIC-wave} a MUSIC-type algorithm is derived to recover small obstacles using reduced time domain data. 
Assuming that the boundary of the inclusion is known, we then turn our attention to the inverse impedance problems of recovering the coefficients from the Dirichlet-to-Neumann mapping. To this end, we prove that real and {\color{black}complex} valued coefficients can be uniquely determined from the {\color{black}knowledge of the} Dirichlet-to-Neumann mapping. In our analysis we can reduce the regularity needed for uniqueness in previous works \cite{CK-gibc,CK-GIBC} but we must assume that we have an infinite set of measurements. {\color{black}We also consider the case where the impedance parameters are complex valued which can not be handled with the analysis given in \cite{CK-gibc,CK-GIBC}.} Since iterative methods normally require an initial guest that is sufficiently close to the actual coefficient to prove convergence as well as the high sensitivity of reconstructing the Laplace-Beltrami coefficient we wish to derive a direct algorithm to recover the boundary coefficients. Here we propose a combination of data completion to recover that Cauchy data on the boundary of the inclusion and a linear system of equations derived from the generalized impedance boundary condition to recover the coefficients.

The rest of the sections of this paper are structured as follows. In Section \ref{direct-problem} we rigorously formulate the direct and inverse problem under consideration. We will {\color{black}use} a variational method to prove well-posedness for $L^{\infty}$ coefficients and derive the appropriate functional setting of the inverse problem {\color{black} where the real and imaginary parts of the coefficients are positive}. In Section \ref{inverse-shape} we will analyze the so-called `Current-Gap Operator' to derive an appropriate Sampling Method to recover the inclusion. {\color{black} The Sampling Method here is the well known Factorization Method where we prove that it satisfies operator bounds that give the result.} In Section \ref{inverse-impedance} we discuss the uniqueness of recovering coefficients using the Dirichlet-to-Neumann mapping for either real or complex-valued coefficients. {\color{black}Finally, we give a brief summary and conclusion of the results in Section \ref{end}.}

\section{The Direct and Inverse Problem}\label{direct-problem}
We begin by considering the direct problem associated with the electrostatic imaging of an impenetrable inclusion with a generalized impedance condition. Assume that $D \subset \R^2$ is a simply connected open set with $C^2$-boundary $\Gamma_\text{1}$ with unit outward normal $\nu$. 
Now let $D_0 \subset D$ be (possible multiple) connected  open set with $C^2$-boundary $\Gamma_\text{0}$, where we assume that $\text{dist}(\Gamma_\text{1} , \overline {D}_0)\geq d>0.$ 
Now for the defective material with the impenetrable  inclusion, we define $u$ as the solution to 
\begin{align}
\Delta u=0 \quad \text{in} \quad  D_1= D \setminus \overline{D}_0  \quad \text{with} \quad u  \big|_{\Gamma_\text{1}}= f \quad \text{and} \quad \mathscr{B}(u) \big|_{\Gamma_\text{0}}=0. \label{defective}
\end{align} 
for a given $f \in H^{1/2}(\Gamma_\text{1})$. Here the function $u$ is the electrostatic potential for the defective material and the boundary operator 
\begin{align}
\mathscr{B} ( u ) = \partial_{\nu}  u -   \frac{ \text{d} }{\text{d} s}  {\eta}   \frac{\text{d}  }{\text{d} s}  u +  {\gamma} u   \label{GIBC}
\end{align}
where ${\text{d} }/{\text{d} s}$ is the tangential derivative and $s$ is the arc-length. 
Here we take $\nu$ to be the unit outward normal to the domain $D_1$ and $ \nu \cdot \grad = \partial_\nu $ is the corresponding normal derivative, see Figure \ref{dp-pic}. In the $\R^3$ case the operator $ \frac{\text{d} }{\text{d} s}  {\eta}   \frac{\text{d} }{\text{d} s}$ is replaced by the Laplace-Beltrami operator $\text{div}_{\Gamma_\text{0}} \big(\eta \text{ grad}_{\Gamma_\text{0}} \big)$.  The analysis in Sections \ref{direct-problem} and \ref{inverse-shape} holds in either $\R^2$ and $\R^3$. {\color{black} The uniqueness result in Section \ref{inverse-shape} only holds in $\R^2$ but the algorithm described for recovering the coefficients is also valid in 3-dimension.}   
\begin{figure}[h]
\centering
\includegraphics[scale=0.34]{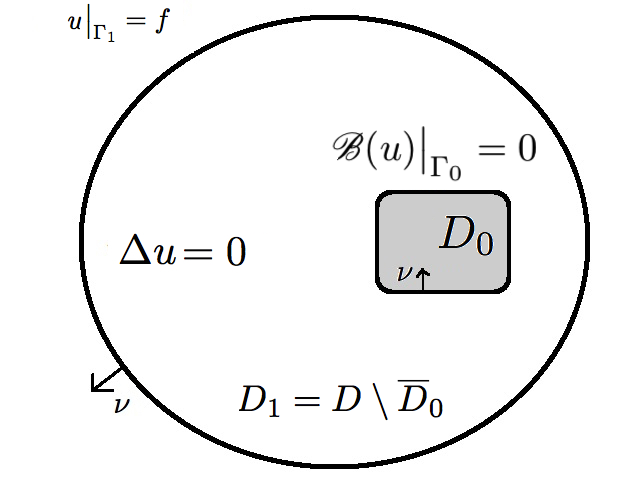}
\caption{ The electrostatic problem for a material with an inclusion. }
\label{dp-pic} 
\end{figure}

In \cite{CK-GIBC} a boundary integral equation method is used to prove the well-posedness of \eqref{defective}$-$\eqref{GIBC} but in their investigation the authors assume that the Dirichlet data $f\in H^{3/2}(\Gamma_1)$ {\color{black} so} that the solution is an $H^2(D_1)$ function. The authors require the impedance parameters to be smooth functions where as we will employ a variational technique which requires less regularity. Now assume that the coefficient $\eta \in L^{\infty }(\Gamma_{0})$ and $\gamma \in L^{\infty }(\Gamma_{0})$. For analytical considerations for the well-posedness of the direct problem {\color{black}and the coming analysis of the inverse problem} throughout the paper we also assume {\color{black}(unless stated otherwise)} that real-parts of the coefficients satisfy 
$$\text{Re}(\eta) \geq \eta_{\text{min}} >0 \quad \text{ and } \quad \text{Re}(\gamma) \geq \gamma_{\text{min}} >0$$
where as the imaginary-parts satisfy 
$${\color{black} \text{Im}(\eta)  > 0 \quad \text{ and } \quad \text{Im}(\gamma) > 0}$$
for almost every $x \in \Gamma_{0}$. Due to the generalized impedance condition \eqref{GIBC} we consider finding 
$u \in \widetilde{H}^1 ( D_1 ) $ that is the solution to \eqref{defective}$-$\eqref{GIBC} for a given $f \in H^{1/2}(\Gamma_\text{1})$. We now define the Hilbert space where we attempt to find the solution as 
$$ \widetilde{H}^1 (D_1) =  \Big\{ u\in {H}^1 (D_1) \quad \text{such that} \quad u \big|_{\Gamma_\text{0}} \in H^1 (\Gamma_\text{0})  \Big\}$$
equipped with the norm
$$ \| \varphi  \|^2_{\widetilde{H}^1( D_1) }= \| \varphi  \|^2_{H^1( D_1) }  +\left\| {\varphi}  \right\|^2_{H^1(\Gamma_\text{0})} $$
and it's corresponding inner-product. 
The boundary condition $\mathscr{B} ( u ) =0$ on $\Gamma_0$ is understood in the weak sense such that 
$$  0 = \int\limits_{\Gamma_{0} } \overline{\varphi} \partial_{\nu} u + \eta\,   \frac{\text{d}  u}{\text{d} s}   \frac{\text{d}   \overline{\varphi} }{\text{d} s} + \gamma  \, u  \overline{\varphi} \, \text{d} s \quad \text{ for all } \quad \varphi \in H^1(\Gamma_0).$$
Notice, that if $f=0$ letting $\varphi=u$ in the line above and applying Green's 1st identity implies that  
$$  \int\limits_{ D_1} | \grad u|^2  \, \text{d} x = - \int\limits_{\Gamma_\text{0} } \eta \left| \frac{\text{d}  u}{\text{d} s} \right|^2 \, \text{d} s  - \int\limits_{\Gamma_\text{0} } \gamma |u|^2 \, \text{d} s.$$
Taking the real part of the above equation gives that $u = 0$ in $D_1$. Since the boundary operator $\mathscr{B}$ is linear we can conclude that \eqref{defective}$-$\eqref{GIBC} has at most one solution.

\begin{lemma}
There exists at most one solution to \eqref{defective}$-$\eqref{GIBC} in $\widetilde{H}^1 (D_1)$. 
\end{lemma}
 
 Given this we now wish to show that the boundary value problem \eqref{defective}$-$\eqref{GIBC} is well posed for any $f \in H^{1/2}(\Gamma_\text{1}).$ To this end, let ${\color{black} u_0} \in H^1(D)$ be the harmonic lifting of the Dirichlet data  where ${\color{black} u_0} = f$ on $\Gamma_\text{1}$ and $\Delta {\color{black} u_0} = 0 $ in $D$. Therefore, by interior elliptic regularity (see for e.g. \cite{evans}) we have that ${\color{black} u_0} \in H^2_{loc}(D)$ which implies that  
\begin{align*}
\|{\color{black} u_0}\|_{H^1(\Gamma_\text{0})} \leq \|{\color{black} u_0}\|_{H^{3/2}(\Gamma_\text{0})} &\leq C \| {\color{black} u_0} \|_{H^2( D_0) }  & \text{Trace Theorem}  \\
													&\leq C \| {\color{black} u_0} \|_{H^1( D ) } & \text{Elliptic Regularity}  \\	
													&\leq C \| f \|_{H^{1/2}(\Gamma_\text{1}) } & \text{Well-Posedness}
\end{align*}
where we have used the continuity of the embedding from $H^{3/2}(\Gamma_\text{0})$ into $H^1(\Gamma_\text{0})$. We now make the ansatz that the solution can be written as $u= v + {\color{black} u_0} \big|_{D_1 }$ with the function $v \in  \widetilde{H}_0^1 (D_1 ,\Gamma_1)$ where we define the space as 
 $$ \widetilde{H}_0^1 (D_1 , \Gamma_\text{1}) =  \Big\{ u\in \widetilde{H}^1 (D_1) \quad \text{such that} \quad u \big|_{ \Gamma_\text{1} } = 0 \Big\}$$
with the same norm as $\widetilde{H}^1 (D_1)$. Now multiplying Laplace's equation by a  test function $\varphi \in \widetilde{H}_0^1 (D_1 ,\Gamma_1)$ and applying integration by parts gives that $v$ solves the variational problem 
\begin{align}
A(v , \varphi ) = -A({\color{black} u_0} , \varphi) \quad \text{ for all } \quad \varphi  \in \widetilde{H}_0^1 (D_1,\Gamma_\text{1}) \label{varform-solu}
\end{align}
where the sequilinear form $A( \cdot \, ,  \cdot ) : \widetilde{H}_0^1 (D_1,\Gamma_\text{1}) \times \widetilde{H}_0^1 (D_1,\Gamma_\text{1}) \longmapsto \C$ is given by 
$$ A(v , \varphi ) =   \int\limits_{ D_1 }  \grad v \cdot \grad \overline{\varphi}  \, \text{d} x + \int\limits_{\Gamma_{0} } \eta\,   \frac{\text{d}  v}{\text{d} s}   \frac{\text{d}  \overline{\varphi}}{\text{d} s}  \, \text{d} s   +  \int\limits_{\Gamma_{0} } \gamma  \, v  \overline{\varphi} \, \text{d} s.$$
It is clear that the sequilinear form is bounded and we have that 
$$ | A( v , v ) | \geq \min \{ 1,\eta_{\text{min}}, \gamma_{\text{min}} \} \, \Big\{ \| \grad v \|^2_{L^2( D_1 ) }  + \left\| v \right\|^2_{H^1 ( \Gamma_{0} )} \Big\}.$$
Since $\widetilde{H}_0^1 (D_1 , \Gamma_\text{1})$ has the Poincar\`{e} inequality due to the zero boundary condition on $\Gamma_1$ we can conclude that $A( \cdot \, ,  \cdot )$ is coercive. By the Lax-Milgram Lemma there is a unique solution $v$ to \eqref{varform-solu} satisfying 
 $$\|v\|_{\widetilde{H}^1( D_1 ) } \leq C \Big\{ \| {\color{black} u_0}  \|_{H^1( D_1) }  +\left\| {\color{black} u_0}  \right\|_{H^1(\Gamma_\text{0})} \Big\} \leq C \| f \|_{H^{1/2}(\Gamma_\text{1}) } $$ 
 where we have used that the sequilinear form is bounded and the regularity estimate for ${\color{black} u_0}$. 
The above analysis gives the following result.  

\begin{theorem}\label{well-posed}
The solution operator corresponding to the boundary value problem \eqref{defective}$-$\eqref{GIBC} $ f \longmapsto u$
is a continuous linear mapping from $H^{1/2}(\Gamma_{1})$ to $\widetilde{H}^1 (D_1 )$.
\end{theorem}

We now assume that the voltage $f $ is applied on the outer boundary $\Gamma_{1}$ and the measured data is give by the current $\partial_\nu u \in H^{-1/2}(\Gamma_{1})$. From the knowledge of the currents we wish to derive a sampling algorithm to determine the impenetrable inclusion $D_0$ without any a prior knowledge of the number of inclusions or the boundary coefficients $\eta$ and $\gamma$. 

We now define the data operator that will be studied in the coming sections to derive the Sampling Method. 
To do so, we {\color{black} recall that $u_0 \in H^1(D)$ is} the unique solution to the following boundary value problem   
\begin{align}
\Delta u_0=0 \quad \text{in} \quad  D \quad \text{with} \quad u_0  \big|_{\Gamma_{1}}= f \label{healthy}
\end{align} 
for a given $f \in H^{1/2}(\Gamma_{1})$. The function $u_0$ is the electrostatic potential for the healthy material and is known since the outer   boundary is known.
Using the linearity of the partial differential equation and boundary conditions on $\Gamma_{0}$ and $\Gamma_{1}$ we have that the voltage to electrostatic potential mappings 
$$ f \longmapsto u_0  \quad \text{and} \quad f \longmapsto u $$
are bounded linear operators from $H^{1/2}(\Gamma_{1})$ to $H^1(D)$ and $\widetilde{H}^1 (D_1)$, respectively. We now define the {\it Dirichlet-to-Neumann } (DtN) mappings such that 
$$\Lambda \, \, \, \text{and} \, \, \,  \Lambda_0 : H^{1/2}(\Gamma_{1}) \longmapsto H^{-1/2}(\Gamma_{1})$$
where 
$$\Lambda f=\partial_{\nu} u \big|_{\Gamma_{1}} \quad \text{ and } \quad \Lambda_0 f=\partial_{\nu} u_0 \big|_{\Gamma_{1}}.$$
By appealing to Theorem \ref{well-posed} and the well-posedness of \eqref{healthy} we have that the DtN mappings are bounded linear operators by Trace Theorems. Our main goal is to solve the {\it inverse shape problem} of recovering the boundary $\Gamma_0$ from a knowledge of the difference of the DtN mappings $(\Lambda_0 - \Lambda)$. This difference is the current gap imposed on the system by the presence of the inclusion $D_0$. By analyzing the data operator $(\Lambda_0 - \Lambda)$ we wish to derive a computationally simple algorithm to detect the inclusion.  

The problem of determining an inclusion and its impedance has been studied by many authors. For the case when $\eta = 0$ iterative methods are analyzed for recovering the inclusion and impedance parameter in \cite{C-Map-imp,EIT-impedance}. In \cite{C-Map-imp} conformal mapping is utilized for the case where there is a single inclusion with $\gamma$ sufficiently small. Where as in \cite{EIT-impedance} a system of non-linear integral equations is used to derive {\color{black}an} iterative scheme to solve the inverse shape and impedance problem. Results on uniqueness and stability for recovering the inclusion and/or impedance has been studied in recent manuscripts \cite{unique-imp,unique-imp2}.  In \cite{unique-imp} it is proven that roughly speaking two Cauchy pairs are enough to uniquely determine the boundary of the inclusion provided the currents are linearly independent and non-negative. To prove the uniqueness the author uses techniques for classical solution to {\color{black}Laplace's} equation which requires that $\gamma$ be a $C^{1,\alpha}$ function and $\Gamma_0$ is class $C^{2,\alpha}$ for some  $0<\alpha <1$. For the case when $\eta \neq 0$ the authors in \cite{2nd-order-inclusion} minimize a cost functional to recover $\eta$ from the measurements assuming that $\gamma$ and $\Gamma_0$ are known. The full inverse shape and impedance problem was studied in \cite{CK-gibc} where non-linear integral equations are used to recover the inclusion and the impedance parameters. The authors also discuss the uniqueness for the inverse problem, where an infinite data set is proven to uniquely recover the inclusion and two Cauchy pairs can recover that impedance parameters once the inclusion is known. 

One disadvantage of using iterative methods is the fact that usually a `good' initial estimate for the inclusion and/or coefficients are needed to insure that the iterative process will converge to the unique solution to the inverse problem. To avoid requiring any a prior knowledge of the physics (boundary conditions of the inclusion) we derive a sampling method to reconstruct the boundary of the inclusion. The idea is that one can split the full inverse problem into two parts: the inverse shape and inverse impedance problems. Once the boundary $\Gamma_0$ is known or approximated via the sampling method recovering the impedance parameters becomes a linear problem and can be solved using a direct algorithm. In the next section, we propose a sampling method to determine that shape and we later remark on how the  impedance parameters can be determined using a direct algorithm.

\section{Solution to the Inverse Shape Problem}\label{inverse-shape}
Now assume that the DtN mapping $\Lambda$ is known from the measurements and $\Lambda_0$ is given from direct calculations. We will give a solution to the inverse shape problem via a sampling method. In general, sampling algorithms connect the support of the inclusion to an indicator function derived from an ill-posed equation involving the measurements operator and a singular solution to the background problem. We now focus on deriving a sampling method for detecting the inclusion $D_0$ from the measurements operator given by the difference of the DtN mappings $(\Lambda_0 - \Lambda)$ for the generalized impedance boundary condition. The work in \cite{FM-GIBC} is for a more generalized boundary condition where Far-Field acoustic data is used. In this section, we will study  the so-called `Current-Gap' operator in order to determine the inclusion.

To begin, we define the auxiliary operator that will be important to deriving our Sampling method. Now, for a given $h \in H^{-1}(\Gamma_0)$ we define $w \in \widetilde{H}^1( D_1)$ to be the unique solution of  
\begin{align}
\Delta  w=0 \quad &\text{in} \quad  D_1  \quad \text{with} \quad w \big|_{\Gamma_1}= 0 \quad \text{and} \quad   \overline{\mathscr{B}}({w})={h} \,\,\,\text{on} \,\,\,\, \Gamma_0 \label{equ-w1}
\end{align} 
where the overline denotes complex conjugation and 
 $$\overline{\mathscr{B}}({w})= \partial_{\nu}  w -   \frac{ \text{d} }{\text{d} s}  \overline{\eta}   \frac{\text{d}  }{\text{d} s}  w +  \overline{\gamma} w.  $$
It is clear that by appealing to a variational argument one can show that \eqref{equ-w1} is well-posed. Therefore, we can define the bounded linear operator
$$ G: H^{-1}(\Gamma_0) \longmapsto H^{-1/2}(\Gamma_1) \quad \text{ given by } \quad Gh = \partial_{\nu} w \big|_{\Gamma_1}$$
where $w$ is the unique solution to equation \eqref{equ-w1}. We first need to understand the operator $G$ where we wish to study the properties of the operator which are given in the following result. 

\begin{theorem}\label{G}
{\color{black} The operator $G$ is} compact and injective.  
\end{theorem}
\begin{proof}
We begin by proving the compactness. Notice that by interior elliptic regularity we can have that the solution to \eqref{equ-w1} is in $H^2_{loc}(D_1)$ for any $h \in H^{-1}(\Gamma_0)$. Since $\text{dist}(\Gamma_\text{1} , \overline {D}_0)\geq d>0$ we have that there is a $\Omega$ such that $D_0 \subset \overline{\Omega} \subset D$ where $\partial \Omega$ is class $C^2$. Notice that the interior $H^2$ regularity implies that the trace of $w$ on $\partial \Omega$ is in $H^{3/2}(\partial \Omega)$ giving that $w$ satisfies $w \in H^2(D \setminus \overline{\Omega})$ by global elliptic regularity. The Trace Theorem gives that $Gh \in H^{1/2}(\Gamma_1)$ and the compact embedding of $H^{1/2}(\Gamma_1)$ into $H^{-1/2}(\Gamma_1)$ proves the compactness. 

To prove the infectivity assume that $h \in \text{Null} (G)$. This gives that the function $w$ is the solution to \eqref{equ-w1} with boundary data $h$ on $\Gamma_0$ has zero Cauchy data on $\Gamma_1$. Therefore, by appealing to Holmgren's Theorem we have that $w=0$ in $D_1$ which  implies that $h=0$.
\end{proof}

Now let the sesquilinear form
$$\langle \varphi  ,\psi  \rangle_{\Gamma_j} = \int\limits_{\Gamma_j } \varphi  \,\overline{ \psi }  \, \text{d}s \quad \text{ for all } \quad \varphi \in H^{p}(\Gamma_j) \quad \text{ and }  \quad  \psi \in  H^{-p}(\Gamma_j)$$ 
denote the dual pairing between $H^{p}(\Gamma_j)$ and $H^{-p}(\Gamma_j)$ (for $p \geq 0$) with $L^2(\Gamma_j)$ as the pivot space where $\Gamma_j$ for $j=0,1$ is the $C^2$ closed curve in $\R^2$ defined in the previous section.
Where we have the inclusions 
$$ H^1(\Gamma_j) \subset H^{1/2}(\Gamma_j) \subset L^{2}(\Gamma_j) \subset H^{-1/2}(\Gamma_j) \subset H^{-1}(\Gamma_j)$$
for the Hilbert Spaces $H^{p}(\Gamma_j)$ and there Dual Spaces $H^{-p}(\Gamma_j)$. 
In our analysis we will need the adjoint operator to $G$ with respect to the above sesquilinear form $\langle \cdot  \, , \cdot  \rangle_{\Gamma_j}$ which is given in the following. 
\begin{theorem}\label{G-dual}
The adjoint operator 
$$G^{*}: H^{1/2}(\Gamma_1) \longmapsto H^{1}(\Gamma_0) \quad  \text{is given by} \quad G^{*} f = - u \big|_{\Gamma_0}.$$
Moreover, $G^{*}$ is injective (i.e. $G$ has a dense range).
\end{theorem}
\begin{proof}
To prove the result we apply Green's 2nd Theorem to the functions $\overline{w}$ and $u$ to obtain 
\begin{align*}
0 &= \int\limits_{\Gamma_1} \overline{w} \dnu u-u \dnu \overline{w} \, \text{d}s + \int\limits_{\Gamma_0} \overline{w} \dnu u-u \dnu \overline{w} \, \text{d}s. 
\end{align*}
We now apply the boundary conditions on $\Gamma_0$ and $\Gamma_1$ which gives that 
\begin{align*}
\int\limits_{\Gamma_1} f \dnu \overline{w} \, \text{d}s &= \int\limits_{\Gamma_0} \overline{w} \dnu u-u \left( \frac{ \text{d} }{\text{d} s}  {\eta}   \frac{\text{d}  }{\text{d} s}  \overline{w} -  {\gamma} \overline{w} + \overline{h}  \right) \, \text{d}s= - \int\limits_{\Gamma_0} u \overline{h} \, \text{d}s.
\end{align*}
The above equality implies that $G^{*} f = - u \big|_{\Gamma_0}$ since 
$$\langle f  ,Gh  \rangle_{\Gamma_1} = \langle G^{*} f  ,h  \rangle_{\Gamma_0} = - \int\limits_{\Gamma_0} u \overline{h} \, \text{d}s.$$
Proving the first part of the result. Now assume $f$ is such that $G^{*} f=0$ and the generalized impedance boundary condition gives that $u$ has zero Cauchy data on $\Gamma_0$. By unique continuation and the Trace Theorem we have that $f=0$ and since $G^*$ is injective we have that $G$ has a dense range (see for e.g. \cite{McLean}). 
\end{proof}

In general, sampling methods connect the region of interest to an ill-posed equation involving the data operator. To do so, one needs a singular solution to the background equation i.e. the equation where the region of interest is not present. Using the singularity of the aforementioned solution to the background problem one shows that an associated ill-posed problem is not solvable unless the singularity is contained in the region of interest. To this end, let $\mathbb{G} (\cdot \, ,z) \in H^1_{loc} \big(D \setminus \{z\} \big)$ for $z \in D$ be the solution to 
$$ \Delta  \mathbb{G} (\cdot \, , \, z) =- \delta(\cdot - z)  \quad \text{in} \quad  D \quad  \text{and} \quad  \mathbb{G} (\cdot \, ,  z) =0 \quad \text{on} \quad  \Gamma_1.$$
The following result shows that Range$(G)$ uniquely determines the region $D_0$.

\begin{theorem}\label{range}
{\color{black} The operator $G$ is} such that 
$\partial_{\nu} \mathbb{G} (\cdot \, ,  z) \in \text{Range}(G)$ if and only if $z \in D_0$.  
\end{theorem}
\begin{proof}
To begin, notice that $\mathbb{G}( \cdot \, ,z)$ is harmonic in $D\setminus \{ z \}$ and interior elliptic regularity implies that
$$\text{ for all $z \in D$} \quad \mathbb{G}( \cdot \, ,z) \in H^{2}_{loc} \big(D\setminus \{ z \} \big).$$
Now assume that $z \in D_0$ and therefore we have that $\mathbb{G}( \cdot \, ,z)$ is a solution to \eqref{equ-w1} in $\widetilde{H}^1( D_ 1)$ with $h_z=\overline{\mathscr{B}} \big( \mathbb{G} (\cdot \, ,  z) \big) \in H^{-1}(\Gamma_0)$. By the definition of $G$ we conclude that $Gh_z=\partial_{\nu} \mathbb{G} (\cdot \, ,  z) \big|_{\Gamma_1}$. 

To prove the remaining implication we proceed by way of contradiction. To this end, assume that $z \in D_1$ and let $h_z$ be such that $Gh_z=\partial_{\nu} \mathbb{G} (\cdot \, ,  z) \big|_{\Gamma_1}$ and by definition this implies that there is a $w_z \in H^1(D_1)$ solving \eqref{equ-w1} such that 
$$w_z= \mathbb{G} (\cdot \, ,  z)=0 \quad \text{and} \quad\partial_{\nu} w_z=\partial_{\nu} \mathbb{G} (\cdot \, ,  z) \quad \text{on} \quad {\Gamma_1}.$$
Notice that $w_z-\mathbb{G}(\cdot \, , z)$ is harmonic in $D_1 \setminus \{z\}$ and has zero Cauchy data on $\Gamma_1$.   
By Holmgren's Theorem we can conclude that $w_z= \mathbb{G} (\cdot \, ,  z)$ in $D_1 \setminus \{z\}$. Since $w_z$ is harmonic in $D_1$ we have that $w_z$ is continuous at the point $z$ (by interior regularity and Sobolev embedding) which gives that $|\mathbb{G} (x,z)|$ is bounded as $x \rightarrow z$, proving the claim by contradiction since $\mathbb{G} (x,z)$ has a logarithmic singularity at $x=z$. Similarly if $z \in \Gamma_0$ we have that there is a $w_z$ that is harmonic in $D_1$ such that $w_z= \mathbb{G} (\cdot \, ,  z)$ in $D_1$. The Trace Theorem would then imply that the $H^{1/2}(\Gamma_0)$ norm of $\mathbb{G} (\cdot \, ,  z)$ is bounded, which again leads to a contradiction. 
\end{proof}


We have shown that the range of the auxiliary operator $G$ uniquely determines the inclusion $D_0$. Our task is to connect the range of $G$ to the range of a known operator defined from $(\Lambda_0 - \Lambda)$. Now we wish to show that there is a positive compact operator $(\Lambda_0 - \Lambda)_{\sharp}$ defined from the knowledge of $(\Lambda_0 - \Lambda)$ such that 
$$ \text{Range}(G) =   \text{Range} \left((\Lambda_0 - \Lambda)_{\sharp}^{1/2} \right).$$
This gives that the support of the inclusion is connected to the range of a compact operator which is known and therefore an indicator function can be derived from the sampling method by solving an ill-poesd problem $(\Lambda_0 - \Lambda)_{\sharp}^{1/2}$ which will only require the singular values and functions of a known operator. We first need to study the current-gap operator.

\begin{theorem}\label{current-gap}
The Current-Gap operator given by $(\Lambda_0 - \Lambda)f=\partial_{\nu} (u_0-u ) \big|_{\Gamma_1}$
where $u_0$ and $u$ are the solutions of \eqref{defective} and \eqref{healthy} is {\color{black}compact and injective}. Moreover, we have the identity 
$$\big\langle f , (\Lambda_0 - \Lambda) f \big\rangle = \int\limits_{D} |\grad  u_0 |^2 \, \text{d}x - \int\limits_{D_1} \grad |u|^2 \text{d}x  - \int\limits_{\Gamma_0 }  \overline{\eta}  \left| \frac{\text{d} }{\text{d} s} u  \right|^2 + \overline{\gamma} |u|^2 \, \text{d} s.$$
\end{theorem}
\begin{proof}
The injectivity follow similarly to the proof in Theorem \ref{G}. Indeed, notice that the difference of the electrostatic potentials $u_0 - u$ in $\widetilde{H}^1( D_1 )$  satisfies the boundary value problem 
\begin{align*}
\Delta  (u_0-u )=0 \quad &\text{in} \quad  D_1  \\ 
(u_0-u)  \big|_{\Gamma_1}= 0 \quad &\text{and} \quad  \mathscr{B}(u_0-u ) = \mathscr{B}(u_0) \,\,\,\text{on} \,\,\,\, \Gamma_0. 
\end{align*} 
Proceeding as in Theorem \ref{G} implies that $(\Lambda_0 - \Lambda)$ maps into $H^{1/2}(\Gamma_1)$ and the compact embedding gives that the operator is compact. 

To prove injectivity let $f \in \text{Null} (\Lambda_0 - \Lambda)$ then we have that $u_0-u$ has vanishing Cauchy data on $\Gamma_1$ and is harmonic in $D_1$. Therefore{\color{red},} by Holmgren's Theorem we can conclude that $\mathscr{B}(u_0) =0$ on $\Gamma_0$.  Since $u_0$ is harmonic in $D_0$ and satisfies the generalized impedance condition we obtain that 
$$\Delta u_0 = 0 \quad \text{in} \quad D_0 \quad \text{and} \quad  \partial_{\nu} u_0 - \frac{ \text{d} }{\text{d} s} {\eta} \frac{\text{d}  }{\text{d} s} u_0 +\gamma u_0 = 0 \quad \text{on} \quad \Gamma_0 $$
where $\nu$ is the unit inward pointing norm to $D_0$. Therefore, we have that 
$$\int\limits_{ D_0} | \grad u_0 |^2  \, \text{d} x =  \int\limits_{\Gamma_\text{0} } \eta \left| \frac{\text{d} }{\text{d} s} u_0 \right|^2 \, \text{d} s  + \int\limits_{\Gamma_\text{0} } \gamma |u_0|^2 \, \text{d} s.$$
Notice, that {\color{black}since Im$(\gamma)$ is strictly positive on $\Gamma_0$} then by taking the imaginary part of the above equality gives that $u_0=0$ {\color{black} on $\Gamma_0$}. The generalized impedance condition implies that $u_0$ has zero Cauchy data {\color{black}on} $\Gamma_0$. By appealing to Holmgren's Theorem and unique continuation we have that $u_0=0$ in $D$ which gives   $f=0$, proving injectivity. 
  
Now by definition we have that 
\begin{align*}
\big\langle f , (\Lambda_0 - \Lambda) f \big\rangle_{\Gamma_1}&= \int\limits_{\Gamma_1} f\,  \partial_{\nu} \overline{u}_0 -  f \, \partial_{\nu} \overline{u} \, \text{d}s = \int\limits_{\Gamma_1} u_0 \partial_{\nu} \overline{u}_0 \, - \,u \partial_{\nu} \overline{u} \, \text{d}s. 
\end{align*}
Now, by Green's 1st identity  we have that 
\begin{align*}
\big\langle f , (\Lambda_0 - \Lambda) f \big\rangle_{\Gamma_1}= \int\limits_{D} | \grad {u}_0 |^2 \, \text{d}x - \int\limits_{D_1} |\grad u|^2\, \text{d}x + \int\limits_{\Gamma_0} u \partial_{\nu} \overline{u} \, \text{d}s.
\end{align*}
From the generalized impedance boundary condition on $\Gamma_0$ we obtain that 
\begin{align*}
\big\langle f , (\Lambda_0 - \Lambda) f \big\rangle_{\Gamma_1}=\int\limits_{D} | \grad {u}_0 |^2 \, \text{d}x - \int\limits_{D_1} |\grad u|^2\, \text{d}x - \int\limits_{\Gamma_0 }  \overline{\eta}\,  \left| \frac{\text{d} }{\text{d} s} u \right|^2  \, + \, \overline{\gamma} |u|^2 \, \text{d} s
\end{align*}
proving the claim.    
\end{proof}

We are almost ready to prove the main result of this section. To do so, we first define the imaginary part of the current-gap operator defined as 
$$\text{Im}(\Lambda_0-\Lambda) = \frac{1}{2 \text{i}} \big[ (\Lambda_0-\Lambda) - (\Lambda_0-\Lambda)^*  \big].$$
It is clear that by Theorem \ref{current-gap} that 
$${\displaystyle \text{Im}\big\langle f , (\Lambda_0 - \Lambda) f \big\rangle = \int\limits_{ \Gamma_0 }  \text{Im}(\eta) \left| \frac{\text{d} }{\text{d} s} u  \right|^2 + \text{Im}(\gamma) |u|^2 \, \text{d} s.}$$
Now assume that  the imaginary parts of ${\eta } \, \text{ and } \,{\gamma} $ are bounded below, then we have that there are constants constant $C_1 \, , \, C_2>0$ such that 
$$ C_1 \| G^{*} f  \|_{H^1(\Gamma_0)}^2 \leq  \text{Im}\big\langle f , (\Lambda_0 - \Lambda) f \big\rangle  \leq C_2 \| G^{*} f  \|_{H^1(\Gamma_0)}^2.$$
This implies that $\text{Im}(\Lambda_0-\Lambda)$ is a positive compact operator by the compactness of $(\Lambda_0-\Lambda)$ and the injectivity of $G^*$. We then have that $\text{Im}(\Lambda_0-\Lambda)$ has a positive square root such that 
$$C_1 \| G^{*} f  \|_{H^1(\Gamma_0)}^2 \leq   \left\| \text{Im}(\Lambda_0 - \Lambda)^{1/2} f \right\|^2_{H^{-1/2}(\Gamma_1)}  \leq C_2 \| G^{*} f  \|_{H^1(\Gamma_0)}^2$$
for all $f \in H^{1/2}(\Gamma_1)$. We now state the quintessential lemma to prove the range equality need for our sampling method. The proof for this lemma can be obtain by the results found in \cite{range-lemma} and the arguments in \cite{FM-EIT} for real Hilbert Spaces can easily be generalized to Banach Spaces. 

\begin{lemma}\label{Range-Equality}
Let $A_j$ be bounded linear operators mapping  $X_j \longmapsto Y$  where $X_j$ and $Y$ are Banach Spaces. If 
$$ \exists \, c_1 \, , \, c_2 >0 \quad \text{such that } \quad c_1 \| A_1^* f \|_{X_1^*} \leq \| A_2^* f \|_{X_2^*}  \leq c_2 \|A_1^* f\|_{X_1^*}  $$
for all $f \in Y^*$ then Range$(A_1)$=Range$(A_2)$
\end{lemma}

By the above inequalities and Lemma \ref{Range-Equality} we have the following result.

\begin{theorem}\label{range-equal}
Assume that the imaginary parts of ${\eta } \, \text{ and } \,{\gamma} $ are bounded below then $$\text{Range}(G) = \text{Range}\left( \text{Im}(\Lambda_0-\Lambda)^{1/2} \right).$$
\end{theorem}

By appealing to Theorem \ref{range} and Theorem \ref{range-equal} we can finally state the main result of the section.This allows one to uniquely determine the inclusion from the knowledge of the DtN mapping $\Lambda$.  

\begin{theorem}\label{FM}
Provided that the imaginary parts of ${\eta} \, \text{ and } \,{\gamma} $ are bounded below then  
$$\partial_{\nu} \mathbb{G} (\cdot \, ,  z) \in \text{Range}\left( \text{Im}(\Lambda_0 - \Lambda)^{1/2} \right) \quad \text{ if and only if}  \quad z \in D_0.$$
Moreover, the mapping $D_0  \longmapsto \Lambda$ is injective. 
\end{theorem}

We define the current-gap equation to be given by: find $f_z  \in H^{1/2}(\Gamma_1)$ such that 
\begin{equation}
 \text{Im}(\Lambda_0-\Lambda)^{1/2} f_z = \partial_{\nu} \mathbb{G} (\cdot \, ,  z)   \quad \text{ for a fixed } \quad z \in D. \label{CGE} 
\end{equation}
Since $ \text{Im}(\Lambda_0-\Lambda)^{1/2}$ has a dense range there exists a sequence of regularized solutions $\big\{f_{z , n} \big\}_{n \in \N}  \in H^{1/2}(\Gamma_1) $ satisfying  
$$\lim\limits_{n \to \infty } \big\|  \text{Im}(\Lambda_0-\Lambda)^{1/2} f_{z , n} - \partial_{\nu} \mathbb{G} (\cdot \, ,  z)   \big\|_{H^{-1/2}(\Gamma_1)}= 0.$$ 
In order to derive our Sampling Method we will show that the sequence $\big\{f_{z , n} \big\}_{n \in \N}$ must be unbounded as $n \to \infty$ for $z \notin D_0$. We proceed by way of contradiction and assume that $\| f_{z , n} \|_{H^{1/2}(\partial D)}$ is bounded with respect to $n$ for any $z \in D$. This implies that there is a weakly convergent subsequence (still denote with $n$) such that $f_{z , n} \weakc f_{z, \infty} \in H^{1/2}(\Gamma_1)$ as $n \to \infty$.
By the compactness of  $\text{Im}(\Lambda_0-\Lambda)^{1/2}$ we can conclude that 
$$  \text{Im}(\Lambda_0-\Lambda)^{1/2} f_{z , n}  \longrightarrow \text{Im}(\Lambda_0-\Lambda)^{1/2} f_{z,\infty} \quad \text{in} \quad H^{-1/2}(\Gamma_1) \quad \text{ as } \quad n \to \infty$$
and therefore $ \text{Im}(\Lambda_0-\Lambda)^{1/2} f_{z,\infty} = \partial_{\nu} \mathbb{G} (\cdot \, ,  z) $ which is a contradiction of Theorem \ref{FM} if $z \notin D_0$. When $z \in D_0$ we have that $\partial_{\nu} \mathbb{G} (\cdot \, ,  z) \in \text{Range}\left( \text{Im}(\Lambda_0 - \Lambda)^{1/2} \right)$ so there exists a $f_{z,\infty}  \in H^{1/2}(\Gamma_1)$ that satisfies \eqref{CGE} and therefore the regularized solutions to \eqref{CGE} is bounded as $n \to \infty$. This implies that at each sampling point $z \in D$ with fixed $n$ we plot the function $z \longmapsto \| f_{z , n} \|^{-1}_{H^{1/2}(\Gamma_1)}$. To approximate the boundary of the inclusion $\Gamma_0$ we construct the level set  $W(z) = \delta \ll 1$.

 {\bf A numerical example for the unit disk:}
For proof of concept we consider applying Theorem \ref{FM} to a simple set up in the unit disk to provide some numerical examples of our inversion method. We will first consider recovering a disk centered at the origin contained in the unit disk. Notice that the Trace Spaces $H^{\pm 1/2}(\Gamma_1)$ can be identified with $H^{\pm 1/2}_{\text{per}}[0,2 \pi]$. To apply Theorem \ref{FM} we need the normal derivative of the Green's function $\mathbb{G} \big( (r,\theta)  , z \big)$ with zero Dirichlet condition on the boundary of the unit disk where we have converted $x$ into polar coordinates $(r,\theta)$. It is well-known that the normal derivative of $\mathbb{G}\big( (r,\theta)  , z \big)$ at $r=1$ is given by the Poisson kernel  
$$  \partial_{r} \mathbb{G} \big( (1,\theta) ,  z \big)  = \frac{1}{2\pi} \frac{1-|z|^2}{|z|^2 +1-2|z| \cos(\theta - \theta_z )}$$
where $\theta_z$ is the polar angle that the point $z$ makes with the positive $x$-axis. 

We now assume $\Gamma_0$ is given by $\rho \big( \cos(\theta) , \sin(\theta)   \big)$ for some constant $ \rho \in (0,1)$.  Since the domain $D$ is assumed to be the unit disk in $\R^2$ we make the ansatz that the electrostatic potential $u(r ,\theta)$ has the following series representation 
\begin{align}\label{series-solution}
u(r,\theta)=a_0 +b_0 \ln r +  \sum_{|n|=1}^{\infty} \big(a_n r^{|n|} + b_n r^{-|n|} \big)  \, \text{e}^{ \text{i} n \theta} \quad \text{in} \quad D_1
\end{align}
which is harmonic in the annular region. The Fourier coefficients $a_n$ and $b_n$ must be determined by the boundary conditions at $r=1$ and $r=\rho$. For simplicity we assume that the coefficients for the generalized impedance condition are constant giving that 
$$ u(1,\theta)=f(\theta) \quad \text{ and } \quad \left( -\frac{\partial}{\partial r} -\frac{\eta}{\rho^2} \, \frac{\partial^2}{\partial \theta^2} + \gamma \right) u (\rho , \theta)=0.$$
We let $f_n$ for $n \in \Z$ be the Fourier coefficients for $f(\theta)$. Notice that the boundary conditions at $r=1$ imply that 
$$  a_0 =f_0 \quad \text{ and } \quad a_n + b_n =f_n \quad \text{ for all } \,\,\, n \neq 0.$$
The boundary conditions at  $r = \rho$ gives that (after some calculations)
$$ b_0 = - \frac{\gamma \rho }{ \gamma \rho \ln \rho- 1}\,  f_0  \quad \text{ and } \quad b_n =\sigma_n \,  a_n  \quad \text{ for all } \,\,\, n \neq 0$$
where 
$$ \sigma_n = \rho^{2 |n|} \frac{|n|\rho -|n|^2 \eta-\gamma \rho^2}{|n|\rho +|n|^2 \eta+\gamma \rho^2} \quad \text{ for all } \,\,\, n \neq 0.$$
This gives that 
$$ a_n = \frac{f_n}{\sigma_n +1}   \quad \text{ and } \quad b_n =  \frac{\sigma_n f_n}{\sigma_n +1} $$
and plugging the sequences into \eqref{series-solution} gives that the corresponding current on the boundary of the unit disk is given by 
\begin{align}
\partial_r u(1,\theta)= - \sigma_0\,  f_0+ \sum_{|n|=1}^{\infty}  |n|  f_n \, \frac{ 1-\sigma_n }{\sigma_n+1} \,  \text{e}^{ \text{i} n \theta} \quad\text{ where } \quad \sigma_0 = \frac{\gamma \rho }{ \gamma \rho \ln \rho- 1}. \label{current-def}
\end{align}

It is clear that the electrostatic potential and subsequent current for the material without an inclusion is given by 
\begin{align}
u_0(r,\theta)={f_0} + \sum_{|n|=1}^{\infty}  {f_n} r^{|n|} \text{e}^{ \text{i} n \theta} \quad \text{and} \quad  \partial_r u_0(1,\theta)= \sum_{|n|=1}^{\infty} |n|  f_n  \text{e}^{ \text{i} n \theta}. \label{potential}
\end{align}
By subtracting \eqref{current-def} from \eqref{potential} gives a series representation of the difference of the DtN mappings. Interchanging the summation and integration we obtain 
$$(\Lambda_0 - \Lambda ) f = \int\limits_{0}^{2 \pi} K(\theta , \phi) f(\phi) \, \text{d} \phi 
\quad \text{where} \quad K(\theta , \phi) = \frac{\sigma_0 }{ 2 \pi} +\frac{1}{\pi} \sum_{|n|=1}^{\infty}  |n|  \frac{ \sigma_n }{ \sigma_n+1 }    \text{e}^{ \text{i} n( \theta - \phi)}. 
$$

We first approximate the kernel function by truncating the series at $|n|=20$. This should be an accurate approximation of the kernel function since $\sigma_n = \mathcal{O} \big(\rho^{2|n|}\big)$ as $|n|$ tends to infinity. We then discretize the truncated integral operator by an equally spaced 64 point Riemann sum approximation and using a collocation method with $64$ equally spaced points $\theta_j \in [0 , 2\pi )$ giving a $64 \times 64$ matrix. We let the matrix ${\bf A}$ represents the discretized operator $(\Lambda_0 - \Lambda )$ and the vector ${\bf b}_z= \big[ \partial_{r} \mathbb{G} \big( (1,\theta_j ) ,  z \big)  \big]_{j=1}^{64}$. Here we add random noise to the discretized matrix ${\bf A}$ such that 
$${\bf A}^\delta =\Big[{\bf A}_{i,j}\big( 1 +\delta \, {\bf E}_{i,j} \big) \Big]_{i,j=1}^{64} \quad \text{ where } \quad  \| {\bf E} \|_2 = 1.$$
{\color{black} The matrix ${\bf E}$ is taken to be the normalized matrix with random entries uniformly distributed between [-1,1].} We take the noise level $\delta = 0.02$ which corresponds to $2 \%$ {\color{black}relative} random noise added to the data {\color{black} in the sense that $\| {\bf A}^\delta - {\bf A}\|_2 \leq \delta \| {\bf A} \|_2$}. We then define the imaginary part and the square root such that 
$$ \text{Im}({\bf A}^\delta) = \frac{1}{2\text{i}} \big[{\bf A}^\delta-({\bf A^\delta})^*\big] \quad \text{ and } \quad \text{Im}({\bf A}^\delta)^{1/2} = {\bf V} {\bf \Sigma}^{1/2} {\bf V^*}$$ 
where ${\bf V \Sigma V^*}$ is the eigenvalue decomposition for the matrix $\text{Im}({\bf A})$. To compute the indicator associated with Theorem \ref{FM} we solve 
$${\bf V} {\bf \Sigma}^{1/2} {\bf V^*} {\bf f}_z = {\bf b}_z$$
 and since the operator is compact  we have that the matrix ${\bf A}$ is ill-conditioned. In order to find an approximate solution to the discretized equation we use Spectral cut-off where the cut-off parameter is taken to be $10^{-8}$. To recover the inclusion we construct 
$$W_{\text{reg}}(z)=\| {\bf f}_z\|_{2}^{-1} \quad \text{where we plot} \quad W(z)=\frac{W_{\text{reg}}(z)}{ \| W_{\text{reg}}(z) \|_{\infty} }.$$
Theorem \ref{FM} implies that $W(z) \approx 1$ for $z \in D_0$ and $W(z) \approx 0$ for $z \notin D_0$.  See Figures \ref{recon1} and \ref{recon2} for reconstructions of this simple example where the function $W(z)$ is used to visualize the defective region. 

\begin{figure}[h]
\hspace{-0.5in}\includegraphics[scale=0.2]{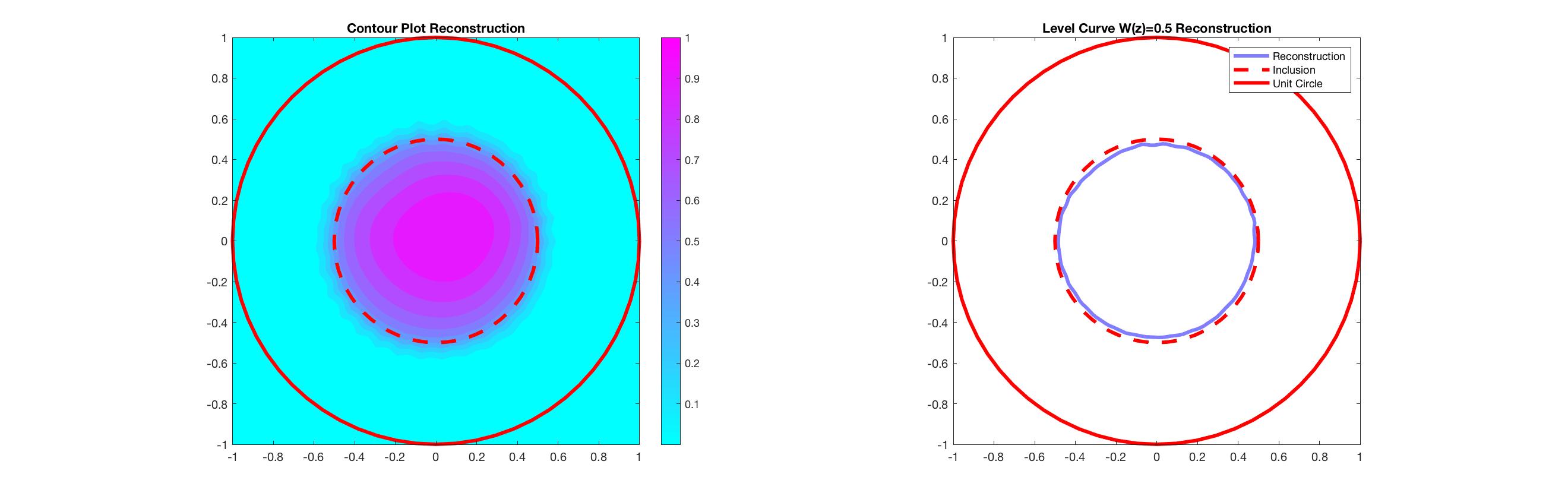}
\caption{ Reconstruction for $\rho = 1/2$ via the Sampling Method with impedance parameters ${\eta = 5+2\text{i}}$ and ${\gamma = 10+\text{i}}$. }
\label{recon1}
\end{figure}

\begin{figure}[h]
\hspace{-0.5in}\includegraphics[scale=0.2]{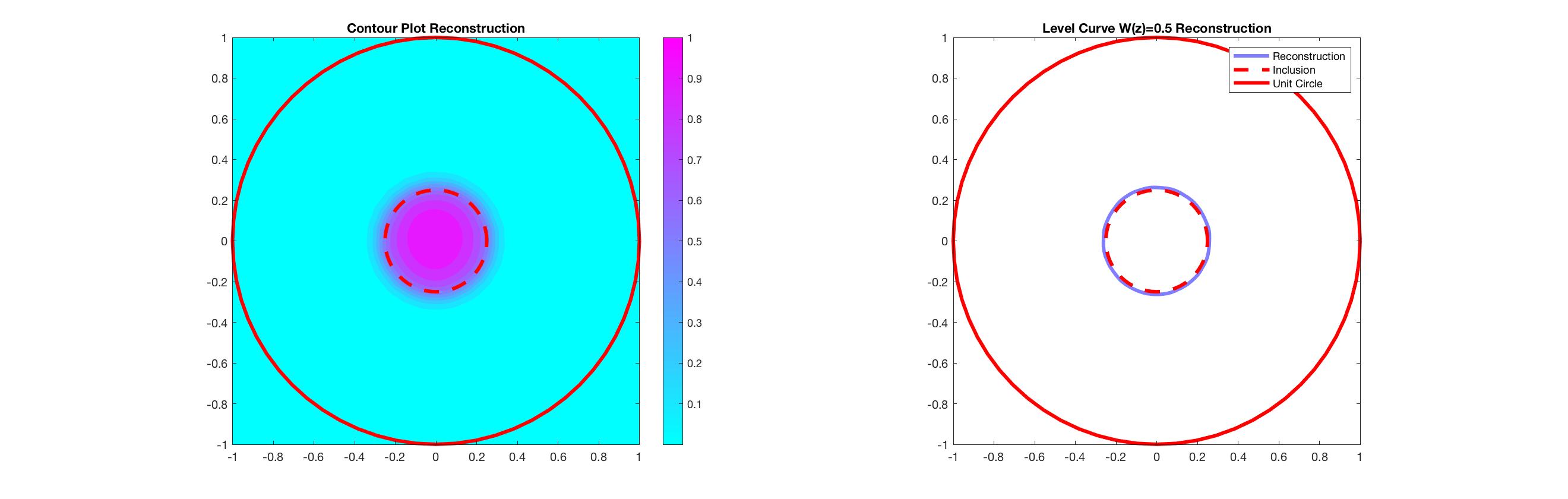}
\caption{Reconstruction for $\rho = 1/4$ via the Sampling Method with impedance parameters ${\eta = 5+2\text{i}}$ and ${\gamma = 10+\text{i}}$.}
\label{recon2}
\end{figure}

In \cite{Harris-Rome} an equivalent indicator function is used given the regularized solution to 
$$\text{Im}({\bf A}) {\bf f}_z = {\bf b}_z  \quad \text{ denoted by}  \quad {\bf f}_z ^\alpha.$$ 
Since the matrix is ill-conditioned one needs to use a suitable regularization techniques to solve the resulting linear system to define the indicator defined in Theorem \ref{FM}. In Section 3.3.3 of  \cite{Harris-Rome} it is shown that on the continue level that the regularized solution ${\bf f}_z ^\alpha$ with regularization parameter $\alpha >0$ is such that 
$$ \lim\inf\limits_{\alpha \to 0} \left\| \text{Im}({\bf A})^{1/2} {\bf f}_z ^\alpha \right\|_{2} < \infty  \quad \text{ if and only if}  \quad z \in D_0.$$
Therefore, we can take the the indicator function to be 
$$P_{\text{reg}}(z)=\| \text{Im}({\bf A}^\delta)^{1/2} {\bf f}_z ^\alpha \|_{2}^{-1} \quad \text{where we plot} \quad P(z)=\frac{P_{\text{reg}}(z)}{ \| P_{\text{reg}}(z) \|_{\infty} }.$$
In order to compute ${\bf f}_z ^\alpha$  we employ the Tikhonov-Morozov regularization strategy. Similarly we expect that $P(z) \approx 1$ for $z \in D_0$ and $P(z) \approx 0$ for $z \notin D_0$.  See Figures \ref{recon3} for the comparison of the indicator functions. 

\begin{figure}[h]
\hspace{-0.5in}\includegraphics[scale=0.2]{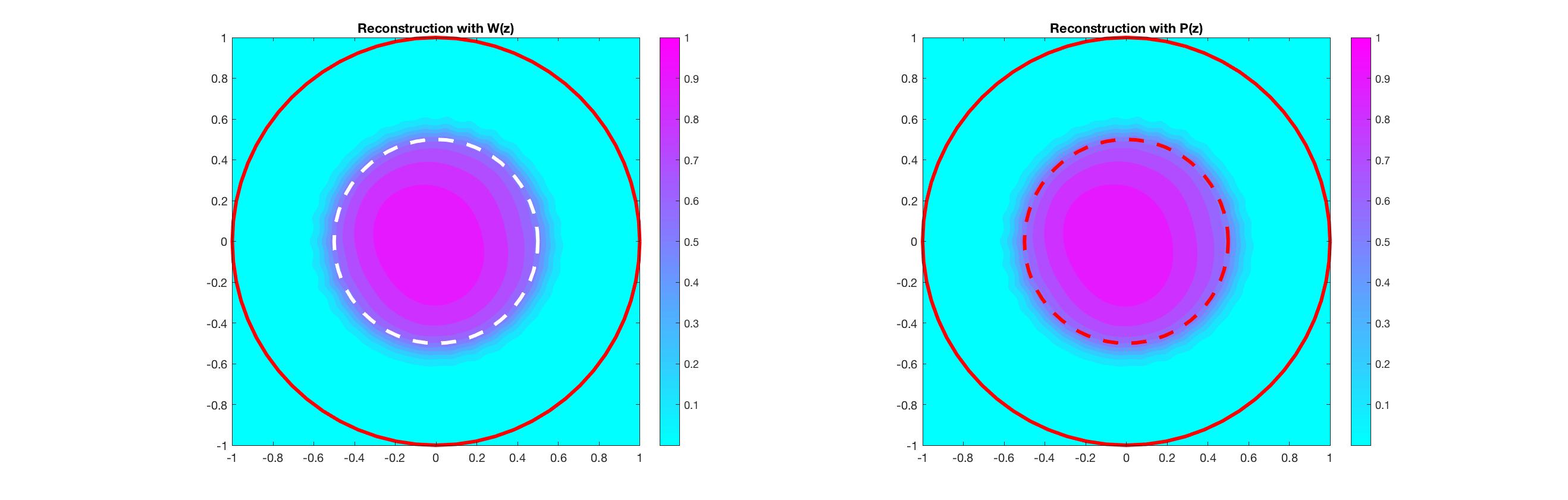}
\caption{ Comparison {\color{black}for the indicator} functions for $\rho = 1/2$ with impedance parameters ${\eta=\gamma = 1+\text{i}}$. }
\label{recon3}
\end{figure}

\section{ Uniqueness of the Inverse Impedance Problem}\label{inverse-impedance}

In this section, we will discuss the inverse impedance problem of determining the material parameters $\eta (x)$ and $\gamma (x)$ from the DtN mapping $\Lambda(\eta , \gamma)$. We will prove uniqueness for the coefficients given the knowledge of the DtN mapping as well as discuss a direct algorithm to recover the coefficients. Since the boundary condition $\mathscr{B}$ is linear with respect to the coefficients we will derive a ``linear'' algorithm for recovering the coefficients. In this section we will assume that the boundary $\Gamma_0$ is known or approximated by the sampling method presented in the previous section. 

We now turn your attention to proving uniqueness for the inverse impedance problem assuming $\Gamma_0$ is known. Since we assume that we have an infinite data set we should be able to prove uniqueness for sufficiently less regularity than is needed in \cite{CK-gibc}. To prove our uniqueness result we first need that following result.

\begin{theorem}\label{dense-set}
The set 
$$\mathcal{U}=\Big\{ u \big|_{\Gamma_0} \, :  \,  u \in \widetilde{H}^1( D_1)  \, \text{  solving  } \eqref{defective} \,  \text{ for all }  \, f \in  H^{1/2}(\Gamma_1) \Big\}$$
is a dense subspace of $L^2(\Gamma_0)$. 
\end{theorem}
\begin{proof}
Notice that by the linearity of the solution mapping $f \mapsto u$ from $H^{1/2}(\Gamma_1)$ to $\widetilde{H}^1( D_1)$ along with the linearity of the trace operator from $\widetilde{H}^1( D_1)$ to $H^{1}(\Gamma_0) \subset L^2(\Gamma_0)$ implies that the set $\mathcal{U}$ is a linear space. To prove the density of the set it is sufficient to prove that the set $\mathcal{U}^{\perp}$ is trivial. We now let $\phi \in \mathcal{U}^{\perp}$ and assume that $v \in \widetilde{H}_0^1 (D_1,\Gamma_\text{1})$ is the unique solution to 
$$ \Delta v=0 \quad \text{in} \quad D_1 \quad \text{ with } \quad v=0 \,\, \text{ on} \quad \Gamma_1 \quad \text{ and } \quad \mathscr{B}(v) = \overline{\phi} \,\, \text{ on } \quad \Gamma_0.$$
Here $\widetilde{H}_0^1 (D_1,\Gamma_\text{1})$ is as defined in Section \ref{direct-problem} and using a variational technique it can be shown that the problem for $v$ is well-posed. We obtain
\begin{align*}
0& =  \int\limits_{\Gamma_0} u \,  \overline{\phi} \, \text{d}s =  \int\limits_{\Gamma_0} u\, \mathscr{B}(v) \, \text{d}s  \\
&=\int\limits_{\Gamma_0} u \partial_\nu v - v \partial_\nu u \, \text{d}s \quad \text{by the generalized impedance condition} \\
&= - \int\limits_{\Gamma_1} u \partial_\nu v - v \partial_\nu u \, \text{d}s \quad \text{ by Green's 2nd Theorem} \\
&= - \int\limits_{\Gamma_1} f \, \partial_\nu v \, \text{d}s  \quad \text{for all } \quad f \in H^{1/2} (\Gamma_\text{1}).
\end{align*}
By the above equality we can conclude that $v=\partial_\nu v=0$ on $\Gamma_1$ and Holmgren's Theorem implies that $v=0$ in $D_1$. The generalized impedance condition gives that $\phi = 0$, proving the claim. 
\end{proof}

With the above result we have all we need to prove that the DtN mapping uniquely determines real-valued parameters $\eta$ and $\gamma$ provided that $\Gamma_0$ is known. {\color{black} We begin with this cases to prove the uniqueness under less regularity assumptions needed in \cite{CK-gibc,CK-GIBC}. The preceding discussion will deal with the case of complex valued impedance parameters. }

\begin{theorem} \label{unique} 
Assume that $\eta$ and $\gamma$ are real-valued and satisfy the assumptions of Section \ref{direct-problem} then the mapping 
$(\eta , \gamma)  \longmapsto \Lambda$ is injective from $C(\Gamma_0) \times L^{\infty} (\Gamma_0)$ to $\mathcal{L} \big( H^{1/2}(\Gamma_1) \, ,\,  H^{-1/2}(\Gamma_1) \big)$.   
\end{theorem}
\begin{proof}
To prove the claim assume that there are two pairs of coefficients $(\eta_j , \gamma_j ) \in C(\Gamma_0) \times L^{\infty} (\Gamma_0)$ that produce the same DtN mapping $\Lambda^{(j)}$, for $j=1,2$. Now let $u^{(j)}$ be the solution to \eqref{defective} with coefficients $(\eta_j , \gamma_j )$. We can conclude that $u^{(1)}$ and $u^{(2)}$ coincide in $D_1$ for all $f \in H^{1/2}(\Gamma_1)$ since their Cauchy data coincides on $\Gamma_1$. This implies that for $u= u^{(1)}=u^{(2)}$ satisfies 
$$ \partial_{\nu} u - \frac{ \text{d} }{\text{d} s} {\eta_1} \frac{\text{d} }{\text{d} s} u+ \gamma_1 u =\partial_{\nu} u -   \frac{ \text{d} }{\text{d} s} {\eta_2} \frac{\text{d} }{\text{d} s} u  + \gamma_2 u  =0 \quad \text{ on } \, \, \Gamma_0. $$
By subtracting the generalized impedance boundary conditions and integrating over $\Gamma_0$ we obtain that  
$$ 0 = \int\limits_{\Gamma_0}  - \frac{ \text{d} }{\text{d} s} {(\eta_1 -\eta_2)} \frac{\text{d} }{\text{d} s} u+ (\gamma_1-\gamma_2) {\color{black}u} \, \text{d} s =  \int\limits_{\Gamma_0} (\gamma_1-\gamma_2) u \, \text{d} s $$
and by Theorem \ref{dense-set} we can conclude that $\gamma_1 = \gamma_2$ a.e. on $\Gamma_0$.

Now assume that $f \in H^{3/2}(\Gamma_1)$, using similar arguments as in Theorem \ref{G} we have that $u \in H^{3/2}(D_1)$ which implies that $\partial_\nu u \in L^2(\Gamma_0)$ and by the generalized impedance boundary condition we can conclude that 
$$\eta_1 \frac{\text{d} u}{\text{d} s}   \in H^1(\Gamma_0) \quad \text{ which implies that } \quad u \in C^1(\Gamma_0).$$   
Notice that since $\gamma_1 = \gamma_2$ subtracting the generalized impedance boundary conditions gives that 
\begin{align*}
\frac{ \text{d} }{\text{d} s} {(\eta_1 -\eta_2)} \frac{\text{d} }{\text{d} s} u  \, =0  \quad \text{ for all } \, \, \, f \in H^{3/2}(\Gamma_1). 
\end{align*}
{\color{black} Where it is sufficient to assume that $f \in H^{3/2}(\Gamma_1)$ since $H^{3/2}(\Gamma_1) \subset H^{1/2}(\Gamma_1)$ and is dense.} This implies that
$$ {(\eta_1 -\eta_2)} \frac{\text{d} u }{\text{d} s}  =C \quad \text{ for all } \, \, \, f \in H^{3/2}(\Gamma_1)$$
where $C$ is some constant. Now let $x(s): [0,\ell] \mapsto \R^2$ be an $\ell$-periodic $C^2$ representation of the curve $\Gamma_0$ where $\ell$ is the length of the curve. Here we identify $H^1(\Gamma_0)$ with the space $H^1_{\text{per}}[0,\ell]$ of $\ell$-periodic functions. It is clear that $u \big( x(0)\big) = u \big( x(\ell)\big)$ for all real-valued $f  \in H^{3/2}(\Gamma_1)$ and therefore by Rolle's Theorem we can conclude  that the tangential derivative for $u$ is zero for some point on the curve which gives that 
$$ {(\eta_1 -\eta_2)} \frac{\text{d} u}{\text{d} s}  =0 \quad \text{ for all real-valued } \, \, \, f \in H^{3/2}(\Gamma_1).$$ 
Without loss of generality assume that there is some $x^* \in \Gamma_0$ such that $(\eta_1- \eta_2)(x^*) >0$. By continuity we have the there exist $\delta >0$ such that $(\eta_1- \eta_2) >0$ for all $x \in \Gamma_0^\delta = \Gamma_0 \cap B(x^*,\delta)$. We can conclude that 
\begin{align} \label{zero-derivative}
\frac{\text{d} u }{\text{d} s}  =0  \quad \text{ on } \Gamma_0^\delta \quad \text{ for all real-valued } \, \, \, f \in H^{3/2}(\Gamma_1).
\end{align}
Now let $f_1$ and $f_2$ be linearly independent real-valued functions in $H^{3/2}(\Gamma_1)$ which implies that the corresponding $u (f_1)$ and $u (f_2)$ in $C^1(\Gamma_0)$ are linearly independent and the Wronskian 
$$W\big(u(f_1) , u(f_2) \big) = u (f_1) \frac{ \text{d} }{\text{d} s}u (f_2)- u(f_2) \frac{ \text{d} }{\text{d} s}u (f_1)$$
can not vanish on any open subset of $\Gamma_0$ due to the generalized impedance boundary condition. By \eqref{zero-derivative} we have that $W\big(u(f_1) , u(f_2) \big) =0$ on $\Gamma_0^\delta$, which contradicts the linear independence of $f_1$ and $f_2$ proving the claim.  
\end{proof}

Notice that in Theorem \ref{unique} to recover both impedance parameters we require the impedance parameters to be real-valued. It is clear from the proof of Theorem \ref{unique} that one can uniquely determine the parameter $\gamma$ with complex valued coefficients. Also it is clear that the arguments in Theorem \ref{dense-set}  {\color{black} hold} whenever the direct problem \eqref{defective} is well-posed. Therefore, one can easily extend the proof of Theorem \ref{unique} where the Laplacian is replaced with any symmetric elliptic partial differential operator with sufficiently smooth real-valued coefficients. 

For the case of uniqueness for complex valued $\eta$ and $\gamma$ then the electrostatic potential $u$ is complex-valued. {\color{black} Now, assuming that} $\eta$ is a (complex-valued) constant and proceed by way of contradiction {\color{black} then} there are two sets of coefficients that produce the same DtN mapping $\Lambda$. {\color{black} We can then} conclude just as in Theorem \ref{unique} that the {\color{black}two} electrostatic potentials are equal and therefore
$$ \frac{\text{d} u }{\text{d} s}  =C \quad \text{ for all } \, \, \, f \in H^{3/2}(\Gamma_1)$$
where $C$ is some (complex-valued) constant. We again identify $H^1(\Gamma_0)$ with the space $H^1_{\text{per}}[0,\ell]$ of $\ell$-periodic functions where we note that $u \big( x(0)\big) = u \big( x(\ell)\big)$. Now define the real-valued functions  
$$ F(s) = \text{Re} \Big\{ u \big( x(s) \big) \Big\} \quad \text{and} \quad G(s) = \text{Im} \Big\{ u \big( x(s) \big) \Big\}$$
for $0 \leq s \leq \ell$. By appealing to Rolle's Theorem for $F(s)$ and $G(s)$ we have that there {\color{black}is} at least one point where the real and imaginary parts of the tangential derivative of $u$ is equal zero which gives that the tangential derivative is zero for all $x$ on $\Gamma_0$. Now proceeding as in Theorem \ref{unique} we then have the following result. 

\begin{theorem} \label{unique2} 
Assume that $\eta$ and $\gamma$ satisfy the assumptions of Section \ref{direct-problem} and $\eta$ constant, then the mapping 
$(\eta , \gamma)  \longmapsto \Lambda$ is injective from $\C \times L^{\infty} (\Gamma_0)$ to $\mathcal{L} \big( H^{1/2}(\Gamma_1) \, ,\,  H^{-1/2}(\Gamma_1) \big)$.   
\end{theorem}

We now derive a linear algorithm to recover $\eta$ and $\gamma$ given the DtN mapping as well as the inner boundary $\Gamma_0$. Therefore, notice that given any voltage $f \in H^{1/2}(\Gamma_1)$ we can compute the corresponding current $\Lambda f$ on $\Gamma_1$. This implies that we know the Cauchy data of the harmonic function $u$ on the outer boundary. Since $\Gamma_0$ is assumed to be known we can use a  data completion algorithm to recover $u (f)$ and $\partial_\nu u (f)$ on the inner boundary.  Recently in \cite{data-completion} and \cite{Harris-Rundell}  data completion algorithms are derived using boundary integral equations. This implies that the mapping 
$$  \big( f , \Lambda f \big)\big|_{\Gamma_1} \,\,  {\longmapsto} \,\, \Big( u (f) \, ,\, \partial_\nu u (f) \Big) \big|_{\Gamma_0} $$
is known. In order to determine the coefficients we recall that $\mathscr{B}(u)=0$ on $\Gamma_0$ for any $f$. Multiplying the generalized impedance condition by $\overline{u}$ gives that 
\begin{align}
- \int\limits_{\Gamma_0} \overline{u} \partial_\nu u  \, \text{d} s  = \int\limits_{\Gamma_0}  {\eta} \left| \frac{\text{d} u}{\text{d} s}  \right|^2 + \gamma |u|^2 \, \text{d} s  \quad \text{for all } \,\,\, f \in H^{1/2}(\Gamma_1). \label{coeff-equ}
\end{align}
Now assume that 
$${\color{black}\eta\big( x(s) \big)} \approx \sum\limits_{n=1}^{N} \eta_n \Psi^{(1)}_n \big( x(s) \big) \quad \text{ and } \quad {\color{black}\gamma\big( x(s) \big)} \approx \sum\limits_{n=1}^{N} \gamma_n \Psi^{(2)}_n \big( x(s) \big) $$
where $\Psi^{(j)}_n$ are some given  linearly independent functions on $\Gamma_0$ for $j=1,2$. Notice that by taking $f_m$ for $m= 1, \cdots ,M$ then \eqref{coeff-equ} gives a $2N \times M$ linear system of equations to recover $\eta$ and $\gamma$. Here we assume that ${\text{d}  u}/{\text{d} s}$ can be recovered from the values of $u$ on $\Gamma_0$ by a finite difference approximation.  

 {\bf A numerical example for constant coefficients:} Just as in the previous section we will provide a simple example in two dimensions to give proof of concept. To this end, we will use \eqref{coeff-equ} to recover constant coefficients from the electrostatic data for the annulus. Here we consider the same numerical example as in the previous section and use the same notation. In order to uses \eqref{coeff-equ} one needs to recover the electrostatic potential on the known/recovered inner boundary using data completion. Using the series representation for $u(r,\theta)$ in \eqref{series-solution} as well as the fact that 
$$u(1,\theta) = f(\theta) \quad \text{ and } \quad \partial_r u(1,\theta) = g(\theta)$$ 
are known we recover the sequences $a_n$ and $b_n$. Applying the boundary conditions where $r=1$ we can conclude that 
\begin{align}
a_n= \frac{|n| f_n +g_n}{2|n|} \quad \text{ and } \quad b_n= \frac{|n| f_n - g_n}{2|n|}  \quad \text{ and } \quad n \neq 0 \label{a-and-b}
\end{align}
where for $n=0$, we have that $a_0=f_0$ and $b_0=g_0$. Here $f_n$ and $g_n$ denotes the Fourier coefficients for the voltage and current measurements. Therefore, using the sequences $a_n$ and $b_n$ along with \eqref{series-solution} we know the electrostatic potential $u(r,\theta)$ in the annulus. 

In order to use \eqref{coeff-equ} to recover constant coefficients we need two pairs of voltage and current measurements. We take $f(\theta)=\text{e}^{ \text{i} n \theta}$ for $n \neq 0$ and $g(\theta)=\partial_r u(1,\theta)$ where the current is given by \eqref{current-def}.  We add random noise to the `measured' current $g$ such that $g^{\delta} (\theta) = g(\theta)+\delta \text{e}^{ \text{i} p \theta}$ where $p \in \N$. To recover the electrostatic potential we use equations \eqref{series-solution} and \eqref{a-and-b} where the Fourier coefficients are computed with Matlab's build in numerical integrator. In our examples we truncate the series solution in \eqref{series-solution} at $|n|=10$. Since we now have recovered $u(r,\theta)$ we solve a 2-by-2 system of equations given by \eqref{coeff-equ} to recover the constant coefficients $\eta$ and $\gamma$. The linear system is solved using the backslash command since the matrix is small a well-conditioned. In Table \ref{tab1} are our results for reconstruction of $\eta$ and $\gamma$ where $\Gamma_0$ is the boundary of the ball with radius $1/2$ where we assume that $\Gamma_0$ is known a prior (i.e. by reconstruction). We see in both cases that the reconstruction of $\eta$ is better than $\gamma$, further numerical study is needed to determine why this occurs.  
\begin{center}
\begin{table}[ht!]%
\centering
\caption{Reconstruction of constant coefficients for $n=1,2$.\label{tab1}}%
\begin{tabular}{c|c|c|c}
\textbf{Exact} & \textbf{Recon: ($\delta=0.01$ , $p=1$)}  &  \textbf{Recon: ($\delta=0.04$ , $p=4$)}   \\
\hline
$\eta= 5+2\text{i}$ & $\eta= 5.0485 + 1.9044\text{i}$  & $\eta= 5.4441 + 1.8730\text{i}$      \\
\hline
$\gamma= 10+\text{i}$ & $\gamma= 10.2582 + 0.2951\text{i}$  & $\gamma= 8.5754 + 0.1658\text{i}$   \\
\hline
\end{tabular}
\caption{ Reconstruction of constant $\eta$ and $\gamma$ from the electrostatic Cauchy data for the inclusion given by a ball with radius $1/2$}.  
\end{table}
\end{center}

\section{Conclusion and Summary }\label{end}
{\color{black} In this paper, we have derived a Sampling Method for recovering an inclusion as well as a linear system of equations for recovering the generalized impedance boundary condition from electrostatic data. Since we use a Sampling Method instead of an iterative method we do not need any a prior knowledge about the inclusion or boundary condition. This method is similar in flavor to the Factorization Method computationally but the operator considered is not factorized to obtain the results. We have also reduce the regularity assumptions from previous works where we pay the price of requiring the full knowledge of the DtN mapping. There are some simple numerical examples presented in two dimensions but an in-depth study of a more robust computational method to solve the inverse shape and impedance problems are needed. Lastly, extending the uniqueness results for the impedance coefficients to three dimensions is also needed. }


\end{document}